\documentclass[12pt, a4paper]{article}
\usepackage{mathrsfs}
\usepackage{amsfonts}
\usepackage{mathrsfs}

\usepackage{latexsym}
\usepackage{graphicx}
\usepackage{amssymb}

\newtheorem{theorem}{Theorem}[section]
\newtheorem{lemma}[theorem]{Lemma}

\newtheorem{definition}[theorem]{Definition}

\def\l{\lambda}

\def\gcd{{\rm gcd}}

\title{{\bf On the local base set of primitive and nonpowerful signed
digraphs\thanks{Supported by NSFC
(Nos. 11271315, 11171290, 11171728) and Jiangsu Qing Lan Project (2014A).}}}
\author{Guanglong Yu$^{a, b}$\thanks{E-mail addresses:
yglong01@163.com (Yu).}
~ Zhengke Miao$^b$ \thanks{Corresponding author: zkmiao@xznu.edu.cn.}
\\ ~ \\
{\footnotesize $^a$Department of Mathematics, Yancheng Teachers
University, Yancheng, 224002, China}\\
{\footnotesize $^b$Department of Mathematics, Xuzhou Normal
University,  Xuzhou, 221116, China}}
\date{}

\begin{document}
\maketitle

\begin{abstract}
In this paper, we consider the local bases of primitive nonpowerful
sign pattern matrices, show that there are ``{\it gaps}" in the
local base set and characterize some sign pattern matrices with
given local bases.

\bigskip
\noindent {\bf AMS Classification:} 05C50

\noindent {\bf Keywords:} Gap; Primitive and nonpowerful; Signed digraph;
Local base
\end{abstract}

\section{Introduction}

\ \ \ \ In this paper, we permit loops but no multiple arcs in a
digraph. We denote by $V(S)$ the vertex set and denoted by $E(S)$ the
arc set for a digraph $S$. A digraph is called a $signed$ digraph if its each edge is assigned one of the signs $-1$ and $1$.
In a signed digraph, the sign of a directed
walk $W=v_{0}e_{1}v_{1}e_{2}\cdots
e_{k}v_{k}\ (e_{i}=(v_{i-1}$, $v_{i})$, $1\le i\le k)$, denoted by
$\mathrm{sgn}$$(W)$, is $\prod\limits_{ i =1}^k $sgn$(e_{i})$. The underlying graph of a signed digraph $S$, denote by $|S|$, is obtained by replacing the sign of each negative edge (with sign $-1$) with sign $1$.

\begin{definition} \label{de1.02}
A strongly connected digraph $S$ is primitive if there exists a
positive integer $k$ such that for any two vertices $v_{i}, v_{j}$ (not necessarily distinct), there exists a directed walk of
length $k$ from $v_{i}$ to $v_{j}$. The least such $k$ is called the
primitive index of $S$, and is denoted by $\mathrm{exp}(S)$.
\end{definition}

As a result, we know that, in a primitive digraph, there exist the least positive integer $k$ such that
there is a directed
walk of length $t$ from $v_{i}$ to $v_{j}$ for any integer $t\geq l$
is called the local primitive index from $v_{i}$ to $v_{j}$. The least $k$ is called the the local
primitive index from $v_{i}$ to $v_{j}$, denoted by $\mathrm{exp}_{S}(v_{i}, v_{j})$. $\mathrm{exp}_{S}(v_{i})=\max\limits_{v_{j}\in
V(S)}\{\mathrm{exp}_{S}(v_{i}, v_{j})\}$ is called the local
primitive index at $v_{i}$. Therefore,
$\mathrm{exp}(S)=\max\limits_{v_{i}\in
V(S)}\{\mathrm{exp}_{S}(v_{i})\}$.

\begin{definition} \label{de1.4}
Assume that $W_{1}$, $W_{2}$ are two directed walks in signed
digraph $S$. They are called a pair of $SSSD$ walks if they have the
same initial vertex, the same terminal vertex and the same length,
but they have different sign.
\end{definition}

\begin{definition} \label{de1.6}
A signed digraph $S$ is primitive and nonpowerful if there exists a
positive integer $l$ such that for any integer $t\geq l$, there are
a pair of $SSSD$ walks of length $t$ from any vertex $v_{i}$ to any
vertex $v_{j}(v_{i}, v_{j}\in V(S))$. The least such $l$ is called
the base of $S$, denoted by $l(S)$.
\end{definition}

As a result, in a primitive and
nonpowerful signed digraph $S$, for $u, v \in V(S)$, there exists an integer $k$ such that their is a pair of $SSSD$ walks
of length $t$ from $u$ to $v$ for any integer $ t\geq k$. The least such $k$ is called the
local base from $u$ to $v$, denoted by $l_{S}(u, v)$.
$l_{S}(u)=\max\limits_{v\in V(S)}\{l_{S}(u, v)\}$ is called the local base at vertex $u$. Therefore,
$$l(S) = \max\limits_{u\in V(S)}l_{S} (u)= \max\limits_{u,v\in
V(S)}l_{S} (u, v).$$

The primitivity of a digraph have been studied extensively which is closely related
to many other problems in various areas of pure and applied
mathematics (for example, see \cite{BLiu}-\cite{55}). For a primitive and nonpowerful signed digraph, the base always seems being not equal to its primitive index, and studying the base needs more treatment (see \cite{Li}, \cite{ShaoYou}, \cite{YSZ}).
Simultaneously, for a primitive and nonpowerful signed digraph, we find that the local base always seems different from its local primitive index (see \cite{W.M}). In \cite{YSZ}, we find that studying the base or local base of a signed digraph is of
great significance for communication science and for studying the properties of sign matrices.

In this paper, we consider the local bases of primitive nonpowerful
sign pattern matrices. The paper is organized as follows: Section 1
introduces the basic ideas of patterns and their supports; Section 2
introduces series of working lemmas; Section 3 shows that there are
some gaps in the local base set and characterizes some digraphs with
given local bases.

\section{Preliminaries}

\ \ \ \ We first introduce some notations.  We denoted by $L(W)$ the length of a directed walk, and denote by $d(v_{i}, v_{j})$ or $d_{S}(v_{i}, v_{j})$ the distance from $v_{i}$ to $v_{j}$ in signed digraph $S$. We denote by $C_{k}$ or {\it $k$-cycle} a directed cycle with length $k$, and denote by $P_{k}$ a directed path of order $k$. A cycle with even (odd) length is called
an {\it even cycle} ({\it odd cycle}). The length of the shortest
directed cycle in a digraph is called the {\it girth} of this digraph.
When there is no ambiguity, a directed walk, a directed
path or a directed cycle will be called a walk, a path or a cycle. A
walk is called a {\it positive (negative) walk} if its sign is positive (negative). The
union of digraphs $H$ and $G$ is the digraph $G\bigcup H$ with
vertex set $V(G)\bigcup V(H)$
 and arc set $E(G)\bigcup E(H)$. The intersection $G\bigcap H$ of digraphs $H$ and $G$ is defined analogously. If
$p$ is a positive integer and if $C$ is a cycle, then $pC$ denotes
the walk obtained by traversing through $C$ $p$ times. If a cycle
$C$ passes through the end vertex of $W$, $W\bigcup pC$ denotes the
the walk obtained by going along $W$ and then going around the cycle
$C$ $p$ times; $pC\bigcup W$ is similarly defined. We use the
notation $v\stackrel{k}{\longrightarrow}u$
($v\stackrel{k}{\not\longrightarrow}u$) to denote that there exists
a (exists no) directed walk with length $k$ from vertex $v$ to $u$.
For a digraph $S$, let $R_{k}(v)=\{u|$
$v\stackrel{k}{\longrightarrow}u$, $u\in V(S)\}$. For a vertex
subset $T$ in a digraph $S$, let $T\stackrel{k}{\longrightarrow}u$
mean that there exists a $s\in T$ such that
$s\stackrel{k}{\longrightarrow}u$.

For a strongly connected digraph $S$ with order $n$, let $C(S)$
denote the cycle length set.

\begin{definition} \label{de2.1}
Let $\{s_{1}$, $s_{2}$, $\cdots$, $s_{\lambda} \}$ be a set of
distinct positive integers with
 gcd$(s_{1}$, $s_{2}$, $\cdots$, $s_{\lambda})$ = 1. The Frobenius number of $
s_{1}$, $s_{2}$, $\cdots$, $s_{\lambda} $, denoted by $\phi(s_{1},
s_{2}, \cdots, s_{\lambda})$, is the smallest positive integer $m$
such that for all positive integers $k\geq m$, there are nonnegative
integers $a_{i} \ (i = 1, 2, \cdots, \lambda)$ such that $k=\sum
\limits_{i = 1}^ \lambda$ $ a_{i}s_{i}$.
\end{definition}

It is well known that

\begin{lemma}{\bf (\cite{LB})} \label{le2.2}
If gcd$(s_{1}, s_{2}) = 1$, then $\phi(s_{1}, s_{2}) = (s_{1} -
1)(s_{2} - 1)$.
\end{lemma}

From Definition \ref {de2.1}, it is easy to see that $\phi(s_{1},
s_{2},\cdots,s_{\lambda})\leq \phi(s_{i}, s_{j})$ if there exist
$s_{i}, s_{j}\in\{ s_{1}$, $s_{2}$, $\cdots$, $s_{\lambda} \}$ such
that gcd$(s_{i}, s_{j}) = 1$. So if $\min \{s_{i}:1\leq i\leq
\lambda\} = 1$, then $\phi(s_{1}, s_{2},\cdots,s_{\lambda})=0$.

\begin{lemma}{\bf (\cite{Kim})} \label{le2.5}
A digraph $S$ with $C(S)=\{p_{1}, p_{2}, \cdots, p_{t}\}$ is primitive if and only if $S$ is strongly
connected $\gcd(p_{1}, p_{2}, \cdots, p_{t})=1$.
\end{lemma}

For a primitive digraph $S$, suppose $C(S)=\{p_{1}$, $p_{2}$,
$\ldots$, $p_{u}\}$. Let $d_{C(S)}(v_{i}$, $v_{j})$ denote the
length of the shortest walk from $v_{i}$ to $v_{j}$ which meets at
least one $p_{i}$-cycle for each $i$, $i=1$, $2, \cdots, u$. Such a
shortest directed walk is called a $C(S)$-walk from $v_{i}$ to
$v_{j}$. And further, $d_{C(S)}(v_{i})$, $d_{i}(C(S))$ and $d(C(S))$ are
defined as follows: $d_{C(S)}(v_{i})=\max\{d_{C(S)}(v_{i}$,
$v_{j})$: $v_{j}\in V(S)\}$, $d(C(S))=\max\{d_{C(S)}(v_{i}$,
$v_{j})$: $v_{i}$, $v_{j}\in V(S)\}$, $d_{i}(C(S))$ $(1\leq i \leq
n)$ is the $ith$ smallest one in $\{d_{C(S)}(v_{i})| 1\leq i \leq
n\}$, $d_{n}(C(S))=d(C(S))$. In particular, if $C(S)=\{p, q\}$,
$d(C(S))$ can be simply denoted by $d\{p, q\}.$

\begin{lemma}{\bf (\cite{BLiu})} \label{le2.6}
Let $S$ be a primitive digraph of order $n$
  and $C(S)=\{p_{1}$, $p_{2}$, $\ldots$, $p_{u}\}$.
  Then $\mbox{exp}(v_{i},v_{j}) \leq d_{C(S)}(v_{i},v_{j})+\phi(p_{1},p_{2},\ldots,p_{u})$ for
  $v_{i}$, $v_{j}\in V(S)$. And furthermore, we have $\mathrm{exp}(S)\leq d(C(S))+\phi(p_{1},p_{2},\ldots,p_{u})$.
\end{lemma}

\begin{lemma}{\bf (\cite{ShaoYou})} \label{le2.101}
Let $S$ be a primitive
  nonpowerful signed digraph. Then $S$ must contain a $p_{1}$-cycle $C_{1}$
  and a $p_{2}$-cycle $C_{2}$ satisfying one
of the following two conditions:

(1) $p_{i}$ is odd, $p_{j}$ is even and sgn$C_{j}=-1$ ($i,j = 1,2$;
$i\neq j)$.

(2) $p_{1}$ and $p_{2}$ are both odd and sgn$C_{1}=-$sgn$C_{2}$.
\end{lemma}

$C_{1}$, $C_{2}$ satisfying condition (1) or (2) are always called
{\it a distinguished cycle pair}. It is easy to prove that
$W_{1}=p_{2}C_{1}$ and $W_{2}=p_{1}C_{2}$ have the same length
$p_1p_2$ but different sign if $p_{1}$-cycle $C_{1}$
  and $p_{2}$-cycle $C_{2}$ are a
distinguished cycle pair, namely $(\mbox{sgn}C_{1})^{p_{2}} =
-((\mbox{sgn}C_{2})^{p_{1}}).$

\begin{lemma}{\bf (\cite{W.M})}\label{le2.12}
Let $S$ be a primitive
  nonpowerful signed digraph of order $n$ and $u \in
V(S)$. If there exists a pair of $SSSD$ walks with length $r$ from
$u$ to $u$, then $l_{S}(u) \leq \mbox{exp}_{S}(u) + r.$
\end{lemma}

\begin{lemma}{\bf (\cite{W.M})}\label{le2.13}
Let $S$ be a primitive
  nonpowerful signed digraph of order $n$, then we have $l_{S}(k) \leq l_{S}(k-1) +
  1$ for $ 2\leq k \leq n$.
\end{lemma}

Let $D_1$ consists of cycle $(v_{n}$, $v_{n-1}$, $\cdots$ , $v_{2}$,
$v_{1}$, $v_{n})$ and arc $(v_{1}$, $v_{n-1})$ and $D_{2}$ =
$D_{1}\bigcup \{(v_{2}$, $v_{n})\}$. Then we have the next
lemma.

\begin{lemma}{\bf (\cite{W.M})}\label{le2.14.01}
Let $S$ be a primitive nonpowerful signed digraph of order $n$ with
$D_{1}$ as its underlying digraph. Then we have $l_{S}(k) = 2n^{2}
-4n + k + 2$ for $1\leq k \leq n$.
\end{lemma}

\begin{lemma}{\bf (\cite{W.M})}\label{le2.14.02}
Let $S$ be a primitive nonpowerful signed digraph of order $n \geq
3$ with $D_{2}$ as its underlying digraph. Then we have:

(1) If the (only) two cycles of length $n-1$ of $S$ have different
signs, then $$l_{S}(k)\leq\left \{\begin{array}{ll}
  2n^{2}
-2n + k + 1,\ & \ 1\leq k\leq n-1;
\\ n^{2}-n,\ & \ k= n. \end{array}\right.$$

(2) If the (only) two cycles of length $n-1$ of $S$ have the same
sign, then $l_{S}(k) = 2n^{2}-4n + k + 1$ for $1\leq k \leq n$.
\end{lemma}

\begin{lemma}{\bf (\cite{W.M})}\label{le2.14}
Let $S$ be a primitive
  nonpowerful signed digraph with order $n\geq6$ whose underlying
  digraph is neither isomorphic to $D_1$ nor to $D_2$, then $l_{S}(k)\leq 2n^{2}- 6n +k +
  4$ for $1\leq k \leq n$.
\end{lemma}

\begin{lemma}{\bf (\cite{YMS})}\label{le2.14.0}
\ (i) Let $A$ be a primitive
  nonpowerful square sign pattern with order $n\geq6$. If $C(S(A))=\{p, q\} \
  (p<q\leq n, p+q>n)$ and the cycles with the same length have the same sign in $S(A)$, then $p(2q-1)
\leq l(A) \leq 2p(q-1)+n.$

(i) Let $n\geq 6$, and let $p, q$ be integers satisfying $p< q\leq
n, p+q\geq n$ and $\gcd(p, q)=1$. Then there exists a primitive
nonpowerful square sign pattern matrix $A$ with order $n$ such that
$C(S(A))=\{p, q\}$ and $l(A)=k$ for each $k\in [p(2q-1) ,
2p(q-1)+n]$, namely, $$[p(2q-1) , 2p(q-1)+n] \subseteq E_{n}^{l}$$
where $E_{n}^{l}=\{l(A)|A$ is a $n\times n$ primitive nonpowerful
sign pattern matrix $\}$.
\end{lemma}

\begin{lemma} {\bf (\cite{YZ})}\label{th3.5} 
Let $S$ be a primitive
  nonpowerful signed digraph of order $n\geq6$. If there exists some $k\ (1\leq k\leq n)$ such that
   $l_{S}(k)\geq \displaystyle
\frac{3}{2}n^{2}-3n+k+4$, then we have the results as follows:

(i)$|C(S)|=2$. Suppose $C(S)=\{p_{1},p_{2}\}(p_{1}<p_{2})$, then
$\gcd(p_{1},p_{2})=1, p_{1}+p_{2}>n;$

(ii) In $S$, all $p_{1}-$cycles have the same sign, all
$p_{2}-$cycles have the same sign, and every pair of $p_{1}-$cycle
and $p_{2}-$cycle form a distinguished cycle pair.
\end{lemma}

\begin{lemma} {\bf (\cite{YZ})}\label{th4.3} 
Let $S$ be a primitive
  nonpowerful signed digraph of order $n\geq6$. If there exists some $k\ (1\leq k\leq n)$ such that
   $l_{S}(k)\geq \displaystyle
\frac{3}{2}n^{2}-3n+k+4$, then
 $$l_{S}(k)\leq \left \{\begin{array}{ll}
 (2n-1)p_{1},\ &
 p_{2}=n, 1\leq k\leq p_{1};\\
\\ (2n-2)p_{1}+k,\ &
 p_{2}=n, p_{1}+1\leq k\leq n;\\
\\ n+2p_{1}(p_{2}-1),\ & p_{2}\leq n-1, 1\leq k\leq n.\end{array}\right.$$
 where
 $C(S)=\{p_{1},p_{2}\}, p_{1}<p_{2}$.
\end{lemma}

\begin{lemma}\label{le2.16}
Let $D_{k, i}$ consists of cycle $C_{n}=(v_{1}$, $v_{n}$, $v_{n-1}$,
$v_{n-2}$, $\ldots$, $v_{2}$, $v_{1})$ and arcs $(v_{1}$,
$v_{n-k})$, $(v_{2}$, $v_{n-k+1})$, $\ldots$, $(v_{i}$,
$v_{n-k+i-1})\ (1\leq i\leq \min\{k+1, n-k-1\})$ (see Fig. 3.1)
where $\gcd(n, n-k)=1$. Then $\mbox{exp}_{D_{k,
i}}(m)=\mbox{exp}_{D_{k, i}}(v_{m})=(n-2)(n-k)+1-i+m$ for $1\leq
m\leq n$.

\end{lemma}

\unitlength 1mm \linethickness{0.4pt}
\begin{picture}(94.33,37.33)
\bezier{156}(41.33,20.00)(39.33,6.67)(65.00,7.33)
\bezier{172}(41.33,20.00)(39.33,36.67)(65.00,35.00)
\bezier{176}(65.33,35.00)(93.33,36.33)(91.33,20.33)
\bezier{168}(91.33,20.33)(92.33,5.67)(65.33,7.33)
\put(61.00,7.33){\line(5,3){30.00}}
\put(84.33,9.33){\line(-2,3){17.11}}
\put(49.33,8.67){\circle*{1.49}} \put(91.00,15.33){\circle*{1.49}}
\put(61.33,7.33){\circle*{1.49}} \put(84.33,9.33){\circle*{1.49}}
\put(91.33,25.33){\circle*{1.33}} \put(67.67,35.00){\circle*{1.49}}
\put(41.67,15.33){\circle*{1.49}} \put(41.33,25.67){\circle*{1.33}}
\put(48.00,9.33){\vector(-2,1){3.00}}
\put(41.33,15.67){\vector(0,1){5.67}}
\put(42.33,28.00){\vector(3,4){1.75}}
\put(52.67,34.67){\vector(1,0){5.33}}
\put(68.00,35.00){\vector(1,0){6.33}}
\put(91.33,25.33){\vector(0,-1){6.67}}
\put(86.00,32.00){\vector(4,-3){3.00}}
\put(90.33,14.33){\vector(-2,-3){1.78}}
\put(74.33,7.33){\vector(-1,0){4.33}}
\put(60.67,7.33){\vector(-1,0){4.67}}
\put(49.33,8.67){\line(6,1){41.67}}
\put(49.33,9.00){\vector(1,0){3.00}}
\put(81.33,19.33){\vector(3,2){4.67}}
\put(77.67,19.67){\vector(-2,3){2.00}}
\put(47.67,6.33){\makebox(0,0)[cc]{$v_{1}$}}
\put(39.00,15.00){\makebox(0,0)[cc]{$v_{n}$}}
\put(37.00,25.67){\makebox(0,0)[cc]{$v_{n-1}$}}
\put(61.33,5.00){\makebox(0,0)[cc]{$v_{2}$}}
\put(85.33,7.00){\makebox(0,0)[cc]{$v_{i}$}}
\put(96.00,15.00){\makebox(0,0)[cc]{$v_{n-k}$}}
\put(98.33,26.00){\makebox(0,0)[cc]{$v_{n-k+1}$}}
\put(67.33,37.33){\makebox(0,0)[cc]{$v_{n-k+i-1}$}}
\put(65.67,0.00){\makebox(0,0)[cc]{Fig. 3.1. $D_{k,i}$}}
\end{picture}

\begin{proof}
It is easy to see $D_{k, i}$ is primitive by lemma \ref{le2.5}.
Also, it is not difficult to check that $$
R_{n-k-(i-2)}(v_{1})\supseteq\left \{\begin{array}{ll}
 \{v_{n}, v_{n-k}, v_{k}\},\ & i=1;\\
\\  \{v_{i-1}, v_{n-k+(i-1)}, v_{k+i-1}\},\ & 2\leq i\leq \min\{k+1,n-k-1\}.
\end{array}\right.$$

If $|\bigcup_{t=1}^{j} R_{t(n-k)-(i-2)}(v_{1})|<n$, we assert that
$$|\bigcup_{t=1}^{j} R_{t(n-k)-(i-2)}(v_{1})\setminus
\bigcup_{t=1}^{j-1} R_{t(n-k)-(i-2)}(v_{1})|\geq 1.$$ Otherwise,
$|\bigcup_{t=1}^{+\infty} R_{t(n-k)-(i-2)}(v_{1})|<n$, which
contradicts that $D_{k, i}$ is primitive.

By the assertion above, we get $|\bigcup_{t=1}^{n-2}
R_{t(n-k)-(i-2)}(v_{1})|=n$ because of $|R_{1}(v_{1})|\geq 3$. So
$\mbox{exp}_{D_{k, i}}(v_{1})\leq (n-2)(n-k)+2-i.$

If $i-1<k$, then $n-k+i-1<n$ and $d(C(D_{k, i}))=d_{C(D_{k,
i})}(v_{n},v_{n-k+i})=n+k-i.$ By Lemma \ref{le2.6}, we get
$$\mbox{exp}(D_{k, i})\leq
d(C(D_{k, i}))+\phi(n,n-k)$$$$=d_{C(D_{k,
i})}(v_{n},v_{n-k+i})+(n-1)(n-k-1)=n+k-i+(n-1)(n-k-1).$$

Now we prove that there is no directed walk of length
$n+k-i+\phi(n,n-k)-1$ from $v_{n}$ to $v_{n-k+i}$. Otherwise,
suppose $W$ is a directed walk of length $n+k-i+\phi(n,n-k)-1$ from
$v_{n}$ to $v_{n-k+i}$. Let $P_{1}$ denote the path from $v_{n}$ to
$v_{n-k+i}$ on cycle $C_{n}$, then $$\mid P_{1}\mid = d(v_{n},
v_{n-k+i})=k-i$$ and $P_{1}$ meet only $n$-cycle not any
$(n-k)$-cycle. $W$ must contain $P_{1}\bigcup C_{n}$, some
$(n-k)-$cycles and some $n$-cycles, namely
$$n+k-i+\phi(n,n-k)-1=k-i+n+a_{1}n+a_{2}(n-k)\ \ (a_{j}\geq
0,\ j=1,2)$$ and $$\phi(n,n-k)-1=a_{1}n+a_{2}(n-k)\ \ (a_{j}\geq 0,\
j=1,2)$$ which contradicts the definition of $\phi(n,n-k)$. So there
is no directed walk of length $n+k-i+\phi(n,n-k)-1$ from $v_{n}$ to
$v_{n-k+i}$, and further, we have
$$\mbox{exp}(D_{k, i})=\mbox{exp}_{D_{k, i}}(v_{n})=\mbox{exp}_{D_{k, i}}(v_{n},v_{n-k+i})=n+k-i+(n-1)(n-k-1)
.$$ Notice that $$
\mbox{exp}_{D_{k, i}}(v_{m})\leq \mbox{exp}_{D_{k, i}}(v_{1})+m-1, 1\leq
m\leq n,$$ $$n+k-i+(n-1)(n-k-1)-((n-2)(n-k)+2-i)=n-1,$$ thus $$\mbox{exp}_{D_{k, i}}(v_{1})\geq
\mbox{exp}_{D_{k, i}}(v_{n})-(n-1)=(n-2)(n-k)+2-i,$$ so we have
$$\mbox{exp}_{D_{k, i}}(v_{1})=(n-2)(n-k)+2-i$$ and
$$\mbox{exp}_{D_{k, i}}(m)=\mbox{exp}_{D_{k, i}}(v_{m})=(n-2)(n-k)+1-i+m\ (1\leq m\leq n).$$

If $k=i-1$, then $n-k+(i-1)=n$, we have $d(C(D_{k, i}))=d_{C(D_{k,
i})}(v_{n},v_{1})=n-1.$ Analogous to the proof of the case
$n-k+(i-1)<n$, we can prove
$$\mbox{exp}(D_{k, i})=\mbox{exp}_{D_{k, i}}(v_{n})=\mbox{exp}_{D_{k, i}}(v_{n},v_{1})=
d_{C(D_{k, i})}(v_{n},v_{1})+\phi(n,n-k)=(n-1)(n-k),$$
$$\mbox{exp}_{D_{k, i}}(v_{1})=(n-2)(n-k)+2-i,$$ and $\mbox{exp}_{D_{k, i}}(m)=
\mbox{exp}_{D_{k, i}}(v_{m})=(n-2)(n-k)+1-i+m$ for $1\leq m\leq n.\
\ \ \ \Box$

\end{proof}

If $n$ is odd, let $\mathscr{L}$ consist of cycle $C_{n}=(v_{1}$,
$v_{n}$, $v_{n-1}$, $v_{n-2}$, $v_{n-3}$, $\ldots$, $v_{2}$,
$v_{1})\ (n\geq 6)$ and arcs $(v_{1}$, $v_{n-2})$, $(v_{3}$,
$v_{n})$. For any positive integer $n$, let $F$ consist of cycle
$C_{n-1}=(v_{1}$, $v_{n}$ , $v_{n-1}$, $v_{n-3}$, $v_{n-4}$,
$\ldots$, $v_{2}$, $v_{1})\ (n\geq 6)$ and arcs $(v_{1}$,
$v_{n-2})$, $(v_{n-2}$, $v_{n-3})$; let $F_{1}$ consist of cycle
$(v_{1}$, $v_{n-1}$, $v_{n-2}$, $\ldots$, $v_{2}$, $v_{1})$ and arcs
$(v_{1}$, $v_{n-2})$, $(v_{2}$, $v_{n})$, $(v_{n}$, $v_{n-1})$; let
$F_{2}$ consist of cycle $(v_{1}$, $v_{n}$, $v_{n-2}$, $v_{n-3}$,
$v_{n-4}$, $\ldots$, $v_{2}$, $v_{1})$ and arcs $(v_{1}$,
$v_{n-2})$, $(v_{n}$, $v_{n-1})$, $(v_{n-1}$, $v_{n-3})$; let
$F_{3}$ consist of cycle $(v_{1} $, $v_{n-2}$, $v_{n-3}$, $v_{n-4}$,
$\ldots$, $v_{2}$, $v_{1})$ and arcs $(v_{1}$, $v_{n-1})$,
$(v_{n-1}$, $v_{n-2})$, $(v_{1}$, $v_{n})$, $(v_{n}$, $v_{n-2})$;
let $F_{i}^{'}$ $(2\leq i\leq n-3)$ consist of cycle $(v_{1}$,
$v_{n-1}$, $v_{n-2}$, $\ldots$, $v_{2}$, $v_{1})$ and arcs $(v_{1}$,
$v_{n-2})$, $(v_{i+1}$, $v_{n})$, $(v_{n}$, $v_{i-1})$; let $F_{4}$
consist of cycle $(v_{1}$, $v_{n-1}$, $v_{n-2}$, $\ldots$, $v_{2}$,
$v_{1})$ and arcs $(v_{1}$, $v_{n-2})$, $(v_{1}$, $v_{n})$,
$(v_{n}$, $v_{n-3})$; let $F_{5}$ consist of cycle $(v_{1}$,
$v_{n-1}$, $v_{n-2}$, $\ldots$, $v_{2}$, $v_{1})$ and arcs $(v_{1}$,
$v_{n-2})$, $(v_{2}$, $v_{n})$, $(v_{n}$, $v_{n-2})$; let $F_{6}$
consist of cycle $(v_{1}$, $v_{n-1}$, $v_{n-2}$, $\ldots$, $v_{2}$,
$v_{1})$ and arcs $(v_{1}$, $v_{n})$, $(v_{n}$, $v_{n-3})$,
$(v_{2}$, $v_{n-1})$; let $F_{7}$ consist of cycle $(v_{1}$,
$v_{n-1}$, $v_{n-2}$, $\ldots$, $v_{2}$, $v_{1})$ and arcs $(v_{1}$,
$v_{n-2})$, $(v_{3}$, $v_{n})$, $(v_{n}$, $v_{n-1})$; let
$\mathscr{B}_{1}$ consist of cycle $C_{n}=(v_{1}$, $v_{n}$,
$v_{n-1}$, $\ldots$, $v_{2}$, $v_{1})$ and arcs $(v_{1}$,
$v_{n-3})$, $(v_{3}$, $v_{n-1})$; let $\mathscr{B}_{2}$ consist of
cycle $C_{n}=(v_{1}$, $v_{n}$, $v_{n-1}$, $\ldots$, $v_{2}$,
$v_{1})$ and arcs $(v_{1}$, $v_{n-3})$, $(v_{4}$, $v_{n})$; let
$\mathscr{B}_{3}$ consist of cycle $C_{n}=(v_{1}$, $v_{n}$,
$v_{n-1}$, $\ldots$, $v_{2}$, $v_{1})$ and arcs $(v_{1}$,
$v_{n-3})$, $(v_{2}$, $v_{n-2})$, $(v_{4}$, $v_{n})$; let
$\mathscr{B}_{3}$ consist of cycle $C_{n}=(v_{1}$, $v_{n}$,
$v_{n-1}$, $\ldots$, $v_{2}$, $v_{1})$ and arcs $(v_{1}$,
$v_{n-3})$, $(v_{2}$, $v_{n-2})$, $(v_{4}$, $v_{n})$; let
$\mathscr{B}_{4}$ consist of cycle $C_{n}=(v_{1}$, $v_{n}$,
$v_{n-1}$, $\ldots$, $v_{2}$, $v_{1})$ and arcs $(v_{1}$,
$v_{n-3})$, $(v_{3}$, $v_{n-1})$, $(v_{4}$, $v_{n})$.

\begin{lemma}\label{le2.16.1}
(1) Suppose that $n$ is odd. Then
$\mathrm{exp}_{\mathscr{L}}(k)=\mathrm{exp}_{\mathscr{L}}(v_{k})=(n-1)(n-3)+k-1$
for $1\leq k\leq n$.

(2) $$\mbox{exp}_{F}(m)=\mbox{exp}_{F}(v_{m})=\left
\{\begin{array}{ll}
 n^{2}-5n+7+m,\ & {\mbox {if}}\ 1\leq m\leq n-2;
\\ n^{2}-5n+6+m,\ & {\mbox {if}}\ n-1\leq m\leq n. \end{array}\right.$$

(3) $\mbox{exp}_{F_{1}}(k)=\mbox{exp}_{F_{1}}(v_{k})=n^{2}-5n+6+k$
for $1\leq k\leq n$.

(4) $$\mbox{exp}_{F_{2}}(k)=\mbox{exp}_{F_{2}}(v_{k})=\left
\{\begin{array}{ll}
 n^{2}-5n+7+k,\ & {\mbox {if}}\ 1\leq k\leq n-2;
\\ n^{2}-5n+6+k,\ & {\mbox {if}}\ n-1\leq k\leq n. \end{array}\right.$$

(5) $$\mbox{exp}_{F_{3}}(k)=\mbox{exp}_{F_{3}}(v_{k})=\left
\{\begin{array}{ll}
 n^{2}-5n+6+k,\ & {\mbox {if}}\ 1\leq k\leq n-1;
\\ n^{2}-5n+5+k,\ & {\mbox {if}}\  k=n. \end{array}\right.$$

(6) $$\mbox{exp}_{F_{i}^{'}}(k)=\left \{\begin{array}{ll}
 \mbox{exp}_{F_{i}^{'}}(v_{k})=n^{2}-5n+6+k,\ & {\mbox {if}}\ 1\leq k\leq i;
 \\ \mbox{exp}_{F_{i}^{'}}(v_{n})=n^{2}-5n+5+k,\ & {\mbox {if}}\ k=i+1;
 \\ \mbox{exp}_{F_{i}^{'}}(v_{k-1})=n^{2}-5n+5+k,\ & {\mbox {if}}\ i+2\leq k\leq n. \end{array}\right.$$

(7) $$\mbox{exp}_{F_{4}}(k)=\left \{\begin{array}{ll}
 \mbox{exp}_{F_{4}}(v_{k})=n^{2}-5n+6+k,\ & {\mbox {if}}\ 1\leq k\leq n-2;
 \\ \mbox{exp}_{F_{4}}(v_{n})=n^{2}-4n+4,\ & {\mbox {if}}\ k=n-1;
 \\ \mbox{exp}_{F_{4}}(v_{n-1})=n^{2}-4n+5,\ & {\mbox {if}}\ k=n. \end{array}\right.$$

(8) $$\mbox{exp}_{F_{5}}(k)=\left \{\begin{array}{ll}
 \mbox{exp}_{F_{5}}(v_{k})=n^{2}-5n+6+k,\ & {\mbox {if}}\ 1\leq k\leq n-1;
\\ \mbox{exp}_{F_{5}}(v_{n-1})=n^{2}-4n+5,\ & {\mbox {if}}\ k=n. \end{array}\right.$$

(9) $$\mbox{exp}_{F_{6}}(k)=\left \{\begin{array}{ll}
 \mbox{exp}_{F_{6}}(v_{k})=n^{2}-5n+6+k,\ & {\mbox {if}}\ 1\leq k\leq n-2;
 \\ \mbox{exp}_{F_{6}}(v_{n})=n^{2}-4n+4,\ & {\mbox {if}}\ k=n-1;
 \\ \mbox{exp}_{F_{6}}(v_{n-1})=n^{2}-4n+5,\ & {\mbox {if}}\ k=n. \end{array}\right.$$

(10) $\mbox{exp}_{F_{7}}(k)=\mbox{exp}_{F_{7}}(v_{k})= n^{2}-5n+5+k$
for $1\leq k\leq n$.

(11)
$\mbox{exp}_{\mathscr{B}_{1}}(k)=\mbox{exp}_{\mathscr{B}_{1}}(v_{k})=(n-1)(n-4)+k$
for $1\leq k\leq n$.

(12)
$\mbox{exp}_{\mathscr{B}_{2}}(k)=\mbox{exp}_{\mathscr{B}_{2}}(v_{k})=(n-3)^{2}+n+k-6$
for $1\leq k\leq n$.

(13)
$\mbox{exp}_{\mathscr{B}_{3}}(k)=\mbox{exp}_{\mathscr{B}_{3}}(v_{k})=(n-3)^{2}+n+k-6$
for $1\leq k\leq n$.

(14)
$\mbox{exp}_{\mathscr{B}_{4}}(k)=\mbox{exp}_{\mathscr{B}_{4}}(v_{k})=(n-3)^{2}+n+k-6$
for $1\leq k\leq n$.

\end{lemma}

\begin{proof}
(1) It is not difficult to check that $R_{n-3}(v_{1})= \{v_{2},
v_{4}, v_{n}\}.$ Similar to the proof of Lemma \ref{le2.16}, we can
prove $|\bigcup_{t=0}^{n-3} R_{t(n-2)+n-3}(v_{1})|=n,$ so
$\mathrm{exp}(v_{1})\leq(n-1)(n-3)$ and
$\mathrm{exp}(v_{n})\leq\mathrm{exp}(v_{1})+d(v_{n}, v_{1})\leq
(n-1)(n-2),$ and further, we get
$\mathrm{exp}(v_{n})=\mathrm{exp}(v_{n}, v_{1})=(n-1)(n-2).$ So
$\mathrm{exp}(v_{1})=(n-1)(n-3)$ and
$$\mathrm{exp}(k)=\mathrm{exp}(v_{k})=(n-1)(n-3)+k-1\ (1\leq k\leq
n).$$

(2) It is not difficult to check that
$$ R_{t(n-2)+2}(v_{1})=\left \{\begin{array}{ll}
\{v_{n-1}, v_{n-3}\},\ & t=0;\\
\\ \{v_{n}, v_{n-2}, v_{n-3}\},\ & t=1.
\end{array}\right.$$

Similar to the proof of Lemma \ref{le2.16}, if $|\bigcup_{t=0}^{j}
R_{t(n-2)+2}(v_{1})|<n$, we can prove
$$|\bigcup_{t=0}^{j} R_{t(n-2)+2}(v_{1}) \setminus
\bigcup_{t=0}^{j-1} R_{t(n-2)+2}(v_{1})|\geq 1\ (j\geq3).$$ So we
have $|\bigcup_{t=0}^{n-3} R_{t(n-2)+2}(v_{1})|=n$ because of
$|R_{2}(v_{1})\bigcup R_{(n-2)+2}(v_{1})|= 4$ and
$$\mbox{exp}_{F}(v_{1})\leq (n-3)(n-2)+2.$$
It is easy to check that $d(C(F))=d_{C(F)}(v_{n}, v_{n-1})=n.$ By
Lemma \ref{le2.6}, thus we have
$$\mbox{exp}_{F}(n)=\mbox{exp}_{F}(v_{n})\leq d(C(F))+\phi (n-1,
n-2)=n^{2}-4n+6.$$ Because of
$d_{C(F)}(v_{n-2})=d_{C(F)}(v_{n-2},v_{n-1})=n-1,$ just as the proof
of
 Lemma \ref{le2.16}, we get
$$\mbox{exp}_{F}(n)=\mbox{exp}_{F}(v_{n})=\mbox{exp}_{F}(v_{n},
v_{n-1})=n^{2}-4n+6,$$
$$\mbox{exp}_{F}(v_{n-2})=d_{C(F)}(v_{n-2},v_{n-1})+\phi (n-1,
n-2)=(n-3)(n-2)+n-1,$$ and get $\mbox{exp}_{F}(v_{1})=(n-3)(n-2)+2,$
$$\mbox{exp}_{F}(m)=\mbox{exp}_{F}(v_{m})=\left \{\begin{array}{ll}
 n^{2}-5n+7+m,\ & {\mbox {if}}\ 1\leq m\leq n-2;
\\ n^{2}-5n+6+m,\ & {\mbox {if}}\ n-1\leq m\leq n. \end{array}\right. $$

In a same way, we can prove (3)-(14). $\ \ \ \ \Box$

\end{proof}

\section{Gaps and characterizations of some digraphs with given local bases}

\begin{theorem} \label{th6.1} 
Let $\gcd(n$, $n-k)=1$ and $S_{k, i}$ be a primitive
  nonpowerful signed digraph with underlying digraph $D_{k, i}\ (1\leq i\leq \min\{k+1, n-k-1\})$. If all
$(n-k)$-cycles have the same sign, then $l_{S_{k,
i}}(m)=l_{S_{i}}(v_{m})=(2n-2)(n-k)+1-i+m\ (1\leq m\leq n).$
\end{theorem}

\begin{proof}
Every pair of $(n-k)$-cycle and $n$-cycle form a distinguished cycle
pair because $S_{k, i}$ is a primitive
  nonpowerful signed digraph. By Lemmas \ref{le2.12}, \ref{le2.16}, we get $$l_{S_{k, i}}(v_{1})\leq
\mbox{exp}_{S_{k, i}}(v_{1})+n(n-k)=(2n-2)(n-k)+2-i.$$ Because of
$d(v_{m},\ v_{1})=m-1$, we have $l_{S_{k, i}}(v_{m})\leq l_{S_{k,
i}}(v_{1})+m-1$ for $1\leq m\leq n$.

{\bf Case 1} $i-1<k$, then $n-k+i-1<n$.

Now we prove that there is no pair of $SSSD$ walks of length
$(2n-2)(n-k)-i+n$ from $v_{n}$ to $v_{n-k+i}$.

Otherwise, suppose $W_{1}, W_{2}$ are a pair of $SSSD$ walks with
length $(2n-2)(n-k)-i+n$ from $v_{n}$ to $v_{n-k+i}$. Let $P$ be the
unique path from $v_{n}$ to $v_{n-k+i}$ on cycle $C_{n}$. Then each
$W_{j}\ (j=1, 2)$ must consists of $P\bigcup C_{n}$, some $n$-cycles
and some $(n-k)$-cycles, namely
$$|W_{j}|=(2n-2)(n-k)-i+n=n+k-i+a_{i}n+b_{i}(n-k)\ (a_{j}, b_{j}\geq 0,\ j=1, 2).$$
Because of $\gcd(n$, $n-k)=1$, so $(a_{1}-a_{2}
)n=(b_{2}-b_{1})(n-k),\ n|(b_{2}-b_{1}),\ (n-k)|(a_{1}-a_{2})$, and
then $b_{2}-b_{1}=nx,\ a_{1}-a_{2}=(n-k)x$ for some integer $x$.

We assert $x=0$.

If $x\geq 1$, then $b_{2}\geq n$, thus we have
$$(2n-2)(n-k)-i+n=n+k-i+a_{2}n+(b_{2}-n)(n-k)+n(n-k)$$ and $\phi(n, n-k)-1=a_{2}n+(b_{2}-n)(n-k),$ which
contradicts the definition of $\phi(n, n-k)$. In a same way, we can
get analogous contradiction when $x\leq -1$. Thus the assertion
$x=0$ is proved.

So $W_{1}, W_{2}$ have the same sign because $b_{2}=b_{1},\
a_{1}=a_{2}$ and all $(n-k)-$cycles have the same sign.  This
contradicts $W_{1}, W_{2}$ are a pair of $SSSD$ walks. Thus there
are no pair of $SSSD$ walks of length $(2n-2)(n-k)-i+n$ from $v_{n}$
to $v_{n-k+i}$, and so
$$l(S_{k, i})= l_{S_{k, i}}(v_{n})=(2n-2)(n-k)+1+n-i.$$

Because of $(2n-2)(n-k)+1+n-i-((2n-2)(n-k)+2-i)=n-1,$
$$l_{S_{i}}(v_{m})\leq l_{S_{i}}(v_{1})+m-1(1\leq m\leq n),$$ we get
$l_{S_{i}}(v_{1})\geq l_{S_{i}}(v_{n})-(n-1),$ so
$l_{S_{i}}(v_{1})=(n-2)(2n-k)+2-i,$ and thus we have
$l_{S_{i}}(m)=l_{S_{i}}(v_{m})=(2n-2)(n-k)+1-i+m$ for $1\leq m\leq
n$ by Lemma \ref{le2.13}.

{\bf Case 2} $k=i-1$, then $n-k+i-1=n$.

As the proof of case 1, we can prove there is no pair of $SSSD$
walks of length $(2n-2)(n-k)-i+n$ from $v_{n}$ to $v_{1}$, and
$l_{S_{k, i}}(v_{n})= (2n-2)(n-k)+1+n-i$, $$l_{S_{k, i}}(m)=l_{S_{k,
i}}(v_{m})=(2n-2)(n-k)+1-i+m\ (1\leq m\leq n).$$ $\ \ \ \Box$
\end{proof}

If $n$ is odd, let $\mathscr{T}$ be a primitive nonpowerful signed
digraph with underlying digraph $\mathscr{L}$, in which all
$(n-2)$-cycles have the same sign. For any positive integer $n$, let
$\mathscr{S}_{0}$ be a primitive nonpowerful signed digraph with
underlying digraph $F$, in which all $(n-1)$-cycles have the same
sign, all $(n-2)$-cycles have the same sign; let $\mathscr{S}_{1}$
be a primitive nonpowerful signed digraph with underlying digraph
$F_{1}$, in which all $(n-1)$-cycles have the same sign, all
$(n-2)$-cycles have the same sign; let $\mathscr{S}_{2}$ be a
primitive nonpowerful signed digraph with underlying digraph
$F_{2}$, in which all $(n-1)$-cycles have the same sign, all
$(n-2)$-cycles have the same sign; let $\mathscr{S}_{3}$ be a
primitive nonpowerful signed digraph with underlying digraph
$F_{3}$, in which all $(n-1)$-cycles have the same sign, all
$(n-2)$-cycles have the same sign; let $\mathscr{S}_{4}$ be a
primitive nonpowerful signed digraph with underlying digraph
$F_{4}$, in which all $(n-1)$-cycles have the same sign, all
$(n-2)$-cycles have the same sign; let $\mathscr{S}_{5}$ be a
primitive nonpowerful signed digraph with underlying digraph
$F_{5}$, in which all $(n-1)$-cycles have the same sign, all
$(n-2)$-cycles have the same sign; let $\mathscr{S}_{6}$ be a
primitive nonpowerful signed digraph with underlying digraph
$F_{6}$, in which all $(n-1)$-cycles have the same sign, all
$(n-2)$-cycles have the same sign; let $\mathscr{S}_{7}$ be a
primitive nonpowerful signed digraph with underlying digraph
$F_{7}$, in which all $(n-1)$-cycles have the same sign, all
$(n-2)$-cycles have the same sign; let $\mathscr{S}_{i}$ be a
primitive nonpowerful signed digraph with underlying digraph
$F^{'}_{i}$, in which all $(n-1)$-cycles have the same sign, all
$(n-2)$-cycles have the same sign; let $\mathscr{Q}_{1}$ be a
primitive nonpowerful signed digraph with underlying digraph
$\mathscr{B}_{1}$, in which all $(n-3)$-cycles have the same sign;
let $\mathscr{Q}_{2}$ be a primitive nonpowerful signed digraph with
underlying digraph $\mathscr{B}_{2}$, in which all $(n-3)$-cycles
have the same sign; let $\mathscr{Q}_{3}$ be a primitive nonpowerful
signed digraph with underlying digraph $\mathscr{B}_{3}$, in which
all $(n-3)$-cycles have the same sign; let $\mathscr{Q}_{4}$ be a
primitive nonpowerful signed digraph with underlying digraph
$\mathscr{B}_{4}$, in which all $(n-3)$-cycles have the same sign.

\begin{theorem} \label{th6.1.1} 
(1) $l_{\mathscr{T}}(k)=l_{\mathscr{T}}(v_{k})=2n(n-3)+k+2\ (1\leq
k\leq n).$

(2)
$$l_{\mathscr{S}_{0}}(k)=l_{\mathscr{S}_{0}}(v_{k})=\left \{\begin{array}{ll}
 2n^{2}-8n+9+k,\ & \ 1\leq k\leq n-2;
\\ 2n^{2}-8n+8+k,\ & \ n-1\leq k \leq n. \end{array}\right.$$

(3)
$l_{\mathscr{S}_{1}}(k)=l_{\mathscr{S}_{1}}(v_{k})=2n^{2}-8n+8+k\
(1\leq k\leq n).$

(4)
$$l_{\mathscr{S}_{2}}(k)=l_{\mathscr{S}_{2}}(v_{k})=\left \{\begin{array}{ll}
 2n^{2}-8n+9+k,\ & \ 1\leq k \leq n-2;
\\ 2n^{2}-8n+8+k,\ & \ n-1\leq k \leq n. \end{array}\right.$$

(5)
$$l_{\mathscr{S}_{3}}(k)=l_{\mathscr{S}_{3}}(v_{k})=\left \{\begin{array}{ll}
 2n^{2}-8n+8+k,\ & \ 1\leq k \leq n-1;
\\ 2n^{2}-7n+7,\ & \ k= n. \end{array}\right.$$

(6)
$$l_{\mathscr{S}_{4}}(k)=\left \{\begin{array}{ll}
 l_{\mathscr{S}_{4}}(v_{k})=
2n^{2}-8n+8+k,\ & \ 1\leq k\leq n-2;
\\ l_{\mathscr{S}_{4}}(v_{n})=2n^{2}-7n+6,
\ & \ k=n-1;
\\l_{\mathscr{S}_{4}}(v_{n-1})=2n^{2}-7n+7,
\ & \ k=n. \end{array}\right.$$

(7)
$$l_{\mathscr{S}_{5}}(k)=l_{\mathscr{S}_{3}}(v_{k})=\left \{\begin{array}{ll}
 2n^{2}-8n+8+k,\ & \ 1\leq k \leq n-1;
\\ 2n^{2}-7n+7,\ & \ k= n. \end{array}\right.$$

(8)
$$l_{\mathscr{S}_{6}}(k)=\left \{\begin{array}{ll}
 l_{\mathscr{S}_{6}}(v_{k})=
2n^{2}-8n+8+k,\ & \ 1\leq k\leq n-2;
\\ l_{\mathscr{S}_{6}}(v_{n})=2n^{2}-7n+6,
\ & \ k=n-1;
\\l_{\mathscr{S}_{6}}(v_{n-1})=2n^{2}-7n+7,
\ & \ k=n. \end{array}\right.$$

(9)
$l_{\mathscr{S}_{7}}(k)=l_{\mathscr{S}_{7}}(v_{k})=2n^{2}-8n+7+k\
(1\leq k\leq n).$

(10)
$$l_{\mathscr{S}_{i}}(k)=\left \{\begin{array}{ll}
 l_{\mathscr{S}_{i}}(v_{k})=2n^{2}-8n+8+k,\ & \ 1\leq k\leq i;
\\ l_{\mathscr{S}_{i}}(v_{n})=2n^{2}-8n+8+k,
\ & \ k=i+1;
\\ l_{\mathscr{S}_{i}}(v_{k-1})=2n^{2}-8n+7+k,
\ & \ i+2\leq k\leq n. \end{array}\right.$$

(11)
$l_{\mathscr{Q}_{1}}(k)=l_{\mathscr{Q}_{1}}(v_{k})=2n^{2}-8n+4+k(1\leq
k\leq n).$

(12)
$l_{\mathscr{Q}_{2}}(k)=l_{\mathscr{Q}_{2}}(v_{k})=2n^{2}-8n+3+k\
(1\leq k\leq n).$

(13)
$l_{\mathscr{Q}_{3}}(k)=l_{\mathscr{Q}_{3}}(v_{k})=2n^{2}-8n+3+k\
(1\leq k\leq n).$

(14)
$l_{\mathscr{Q}_{4}}(k)=l_{\mathscr{Q}_{4}}(v_{k})=2n^{2}-8n+3+k\
(1\leq k\leq n).$

\end{theorem}

\begin{proof}
(1) By Lemma \ref{le2.16.1}, we get
$l_{\mathscr{T}}(v_{1})\leq(n-1)(n-3)+n(n-2)$ and
$l_{\mathscr{T}}(v_{n})\leq l_{\mathscr{T}}(v_{1}) +d(v_{n},
v_{1})\leq(2n-1)(n-2).$ As the proof of Case 1 in Theorem
\ref{th6.1}, we can prove
$$l_{\mathscr{T}}(v_{n})=l_{\mathscr{T}}(v_{n}, v_{1})=(2n-1)(n-2),\
l_{\mathscr{T}}(v_{1})=(n-1)(n-3)+n(n-2)$$ and
$l_{\mathscr{T}}(k)=l_{\mathscr{T}}(v_{k})=2n(n-3)+k+2\ (1\leq k\leq
n).$ In a same way, we can prove the Theorems (2)-- (14) $\ \ \
\Box$
\end{proof}

\begin{theorem} \label{th6.11} 
Let $S$ be a primitive nonpowerful signed digraph with order
$n(n\geq 14)$. Then we have:

(1) There is no $S$ such that $l_{s}(k)\in [2n^{2}-8n+10+k,
2n^{2}-4n +k]$ for $1\leq k\leq n-2$ and no $S$ such that
$l_{S}(k)\in [2n^{2}-8n+9+k, 2n^{2}-4n +k]$ for $n-1\leq k\leq n$ if
$n$ is an positive even integer.

(2) If $n$ is an positive odd integer, there is no $S$ such that
$l_{S}(k)\in [2n^{2}-6n+5+k, 2n^{2}-4n +k]$ for $1\leq k\leq n$;

there is no $S$ such that $l_{S}(k)\in [2n^{2}-8n+10+k, 2n^{2}-6n
+k+1]$ for $1\leq k\leq n-2$;

there is no $S$ such that $l_{S}(k)\in [2n^{2}-8n+9+k, 2n^{2}-6n
+k+1]$ for $n-1\leq k\leq n$; and further, we have:

(i) $l_{S}(k)=2n^{2}-6n+4+k(1\leq k\leq n)$ if and only if $|S|\cong
D_{2, 1}$;

(ii) $l_{S}(k)=2n^{2}-6n+3+k(1\leq k\leq n)$ if and only if
$|S|\cong D_{2, 2}$, the cycles with the same length have the same
sign in $S$;

(iii) $l_{S}(k)=2n^{2}-6n+2+k(1\leq k\leq n)$ if and only if
$|S|\cong D_{2, 3}$ or $|S|\cong \mathscr{L}$, the cycles with the
same length have the same sign in $S$.

(3) (i) $l_{S}(k)=2n^{2}-8n+9+k(1\leq k\leq n-2)$ if and only if
$|S|\cong F$ or $|S|\cong F_{2}$, the cycles with the same length
have the same sign in $S$;

(ii) $l_{S}(k)=2n^{2}-8n+8+k(1\leq k\leq n)$ if and only if
$|S|\cong F_{1}$, the cycles with the same length have the same sign
in $S$;

 $l_{S}(k)=2n^{2}-8n+8+k(1\leq k\leq n-2)$ if and only if
$|S|$ is isomorphic to one of $\{F_{1}, F_{3}, F_{4}, F_{5}, F_{6},
F^{'}_{n-3}\}$, the cycles with the same length have the same sign
in $S$;

 $l_{S}(k)=2n^{2}-8n+8+k(1\leq k\leq n-1)$ if and only if
$|S|$ is isomorphic to one of $\{F_{1}, F_{3},  F_{5}\}$, the cycles
with the same length have the same sign in $S$;

$l_{S}(k)=2n^{2}-8n+8+k(n-1\leq k\leq n)$ if and only if $|S|$ is
isomorphic to one of $\{F, F_{1},  F_{2}\}$, the cycles with the
same length have the same sign in $S$.

(iii) $l_{S}(k)=2n^{2}-8n+7+k(1\leq k\leq n)$ if and only if
$|S|\cong F_{7}$, the cycles with the same length have the same sign
in $S$; $l_{S}(k)=2n^{2}-8n+7+k(n-1\leq k\leq n)$ if and only if
$|S|$ is isomorphic to one of $\{F_{4}, F_{6},
F_{7}\}\bigcup\{F^{'}_{i}| 2\leq i\leq n-3\}$, the cycles with the
same length have the same sign in $S$; $l_{S}(k)=2n^{2}-8n+7+k(k=
n)$ if and only if $|S|$ is isomorphic to one of $\{F_{3}$, $F_{4}$,
$F_{5}$, $F_{6}$, $F_{7}\}\bigcup\{F^{'}_{i}| 2\leq i\leq n-3\}$,
the cycles with the same length have the same sign in $S$.

(4) $2n^{2}-8n+6+k$ if and only if $|S|\cong D_{3, 1}$;
$2n^{2}-8n+5+k$ if and only if $|S|\cong D_{3, 2}$, the cycles with
the same length have the same sign in $S$; $2n^{2}-8n+4+k$ if and
only if $|S|\cong D_{3, 3}$ or $|S|\cong \mathscr{B}_{1}$, the
cycles with the same length have the same sign in $S$;
$2n^{2}-8n+3+k$ if and only if $|S|$ is isomorphic to one of
$\{D_{3, 4},  \mathscr{B}_{2}, \mathscr{B}_{3}, \mathscr{B}_{4}\}$,
the cycles with the same length have the same sign in $S$.

(5) For any positive integer $n$, there is no $S$ such that
$l_{s}(k)\in [2n^{2}-9n+13, 2n^{2}-8n +2+k]$ for $1\leq k\leq n$.

\end{theorem}

\begin{proof}
Note that $n\geq 14$, then $2n^{2}-9n+12\geq \displaystyle
\frac{3}{2}n^{2}-3n+k+4.$ By Lemma \ref{th3.5}, then
$C(S)=\{p_{1},p_{2}\}, p_{1}<p_{2}, p_{1}+p_{2}>n$, all the
$p_{1}-$cycles have the same sign, all the $p_{2}-$cycles have the
same sign in $S$. By Lemma \ref{th4.3}, we know that for $1\leq
k\leq n$,
$$l_{S}(k)\leq \left \{\begin{array}{ll}
 (2n-2)p_{1}+n\leq 2n^{2}-9n+8,\ & \ p_{2}=n, p_{1}\leq n-4;
\\ n+2p_{1}(p_{2}-1)\leq 2n^{2}-9n+12,\ & \ p_{1}\leq n-3, p_{2}\leq n-1.\end{array}\right.$$
So, if $l_{S}(k)\geq 2n^{2}-9n+13$, there are just the following
cases:

(1) $p_{2}=n, p_{1}= n-1$;

(2) $p_{2}=n, p_{1}= n-2$;

(3) $p_{2}=n, p_{1}= n-3$;

(4) $p_{2}=n-1, p_{1}= n-2$.

Then the theorem follows from the Lemmas
\ref{le2.14.01}--\ref{le2.14}, Theorems \ref{th6.1}, \ref{th6.1.1}.
$\ \ \ \Box$
\end{proof}

\noindent{\bf Acknowledgment}

Many thanks to the referees for their kind reviews and helpful
suggestions.

\small {

}

\end{document}